\documentclass[12pt,reqno]{amsart}
\usepackage[left=80pt,right=80pt]{geometry}
\usepackage[usenames]{color}

\usepackage{amsmath}
\usepackage{amssymb}
\usepackage{amsthm}
\usepackage{enumitem}
\usepackage{float}
\usepackage{graphicx}
\usepackage{tikz}
\usepackage{cases}
\usepackage{blkarray}
\usepackage{array}
\usetikzlibrary{arrows}
\usepackage{tikz-qtree}
\usetikzlibrary{graphs,graphs.standard,calc}
\usepackage[caption=false]{subfig}
\usepackage{hyperref,url}
\hypersetup{
	colorlinks=true,
  linkcolor=black,          
  citecolor=black,         
  filecolor=black,      
  urlcolor=black           
}

\theoremstyle{plain}
\newtheorem{theorem}{Theorem}

\newtheorem{lemma}[theorem]{Lemma}

\theoremstyle{definition}

\theoremstyle{remark}

\newtheorem{conjecture}[theorem]{Conjecture}

\newcommand{\seqnum}[1]{\href{https://oeis.org/#1}{\rm \underline{#1}}}
\allowdisplaybreaks
\newcommand{\mylabel}[2]{#2\def\@currentlabel{#2}\label{#1}}

\begin{document}
\tikzset{mystyle/.style={matrix of nodes,
        nodes in empty cells,
        row 1/.style={nodes={draw=none}},
        row sep=-\pgflinewidth,
        column sep=-\pgflinewidth,
        nodes={draw,minimum width=1cm,minimum height=1cm,anchor=center}}}
\tikzset{mystyleb/.style={matrix of nodes,
        nodes in empty cells,
        row sep=-\pgflinewidth,
        column sep=-\pgflinewidth,
        nodes={draw,minimum width=1cm,minimum height=1cm,anchor=center}}}

\title{On a conjecture of McNeil}

\author[SELA FRIED]{Sela Fried$^{*}$}
\thanks{$^{*}\,$Department of Computer Science, Israel Academic College,
52275 Ramat Gan, Israel. 
\href{mailto:friedsela@gmail.com}{\tt friedsela@gmail.com}}

\maketitle

\begin{abstract}
Suppose that the $n^2$ vertices of the grid graph $P_n^2$ are labeled, such that the set of their labels is $\{1,2,\ldots,n^2\}$. The labeling induces a walk on $P_n^2$, beginning with the vertex whose label is $1$, proceeding to the vertex whose label is $2$, etc., until all vertices are visited. The question of the maximal possible length of such a walk, denoted by $M(P_n^2)$, when the distance between consecutive vertices is the Manhattan distance, was studied by McNeil, who, based on empirical evidence, conjectured that $M(P_n^2)=n^3-3$, if $n$ is even, and $n^3-n-1$, otherwise. In this work we study the more general case of $P_m\times P_n$ and capture $M(P_m\times P_n)$, up to an additive factor of $1$. This holds, in particular, for the values conjectured by McNeil.
\bigskip

\noindent \textbf{Keywords:} graph labeling, permutation.
\smallskip

\noindent
\textbf{Math.~Subj.~Class.:} 05C78, 05A05.
\end{abstract}

\section{Introduction}
Let $m$ and $n$ be two natural numbers and consider the grid graph $P_m\times P_n$. Assume that the $mn$ vertices of $P_m\times P_n$ are labeled such that the set of their labels is $\{1,2,\ldots,mn\}$. The labeling induces a walk  on $P_m\times P_n$, beginning with the vertex whose label is $1$, proceeding to the vertex whose label is $2$, etc., until all vertices are visited. This work is concerned with the following question: What is the maximal possible length of such a walk, denoted by $M(P_m\times P_n)$, if the distance between consecutive vertices is the Manhattan distance?

This question was studied in two special cases: The case $m=1$ corresponds to standard permutations and it was shown by Bulteau et al.\ \cite{bulteau2021disorders} that $M(P_n)=\lfloor n^2/2\rfloor -1$. The second case, in which $m=n$, was studied by McNeil (see sequence \seqnum{A179094} in  \cite{OL}). Based on empirical evidence, McNeil made the following conjecture.
\begin{conjecture}
Assume that $n\geq 2$. Then $$M(P_n^2)=\begin{cases}n^3-3, &\textnormal{if } n \textnormal{ is even} ;\\n^3-n-1, & \textnormal{otherwise}.\end{cases}$$
\end{conjecture}

Our results are summarized in the following theorem. In particular, it follows from them that the values conjectured by McNeil are exact, up to an additive constant of $1$.

\begin{theorem}\label{thm;ag1}
Assume that $m,n\geq 3$. Then
\begin{enumerate}
\item If $m$ and $n$ are both even then $M(P_m\times P_n)\in\{r,r+1\}$, where $r=mn(m+n)/2 - 3$.
\item If $m$ and $n$ are both odd then $M(P_m\times P_n)\in\{r,r+1\}$, where $r=mn(m+n)/2-(m+n)/2-1$.    
\item If $m$ is odd and $n$ is even then $M(P_m\times P_n)=mn(m+n)/2-n/2-1$. 
\end{enumerate} Furthermore, for every $n\geq 2$, we have $M(P_2\times P_n)\in\{r,r+1\}$, where $r=(n+1)^2-4$.
\end{theorem}

\section{Main results}

Let $m,n\geq 2$ be two natural numbers to be used throughout this work. We write $[n]=\{1,2,\ldots,n\}$ and  identify the vertices of $P_m\times P_n$ with the set $V=\{(i,j)\;:\;i\in [m], j\in [n]\}$. The proof of Theorem \ref{thm;ag1} consists of an upper bound and a lower bound for $M(P_m\times P_n)$. 

\subsection{The upper bound}

A key ingredient in the proof of the upper bound is the following generalization of \cite[Theorem 8]{bulteau2021disorders} to $m$ copies of $[n]$. We omit its proof, which would spread along several pages, but is almost identical to the one provided by \cite{bulteau2021disorders}.

\begin{lemma}\label{lem;ga10}
Let $S(m,n)$ be the set of permutations of the elements of the multiset containing $m$ copies of $[n]$. Then 

\[
\max\left\{\sum_{i=1}^{mn-1}|\sigma_{i+1}-\sigma_i|\;:\;(\sigma_1,\ldots,\sigma_{mn})\in S(m,n)\right\}=
2m\left\lfloor \frac{n}{2}\right\rfloor \left\lceil\frac{n}{2}\right\rceil-1_{n \textnormal{ is even}},
\] where $1_{n \textnormal{ is even}}$ is $1$ if $n$ is even and $0$ otherwise.
\end{lemma}

\begin{lemma}
Let $\sigma\colon V\to [mn]$ be a bijection. Then \[\sum_{t=1}^{mn-1} d( \sigma^{-1}(t),\sigma^{-1}(t+1))\leq 2n\left\lfloor \frac{m}{2}\right\rfloor \left\lceil\frac{m}{2}\right\rceil-1_{m \textnormal{ is even}}+2m\left\lfloor \frac{n}{2}\right\rfloor \left\lceil\frac{n}{2}\right\rceil-1_{n \textnormal{ is even}},\] where $d((i_1,j_1),(i_2,j_2))=|i_1-i_2|+|j_1-j_2|$ is the Manhattan distance between $(i_1,j_1),(i_2,j_2)\in V$.
\end{lemma}

\begin{proof}
For $t\in[mn]$ write $(i(t),j(t))=\sigma^{-1}(t)$ and notice that the multiset $\{i(t)\;:\;t\in[mn]\}$ is equal to $n$ copies of $[m]$ and the multiset $\{j(t)\;:\;t\in[mn]\}$ is equal to $m$ copies of $[n]$. Thus, by Lemma \ref{lem;ga10}, 
\begin{align}
\sum_{t=1}^{mn-1} d( \sigma^{-1}(t),\sigma^{-1}(t+1))&=\sum_{t=1}^{mn-1} |i(t+1)-i(t)|+\sum_{t=1}^{mn-1} |j(t+1)-j(t)|\nonumber\\&\leq 2n\left\lfloor \frac{m}{2}\right\rfloor \left\lceil\frac{m}{2}\right\rceil-1_{m \textnormal{ is even}}+2m\left\lfloor \frac{n}{2}\right\rfloor \left\lceil\frac{n}{2}\right\rceil-1_{n \textnormal{ is even}}.\nonumber\qedhere
\end{align}
\end{proof}

\subsection{The lower bound}

The proof of the lower bound consists of four lemmas, whose easy proofs we omit.

\begin{lemma}\label{lem;222}
Assume that $m=2$ and that $n$ is odd. Let $\sigma\colon V\to [2n]$ be the bijection given by \[\left[\begin{array}{ccccccccc}
n-1 & \cdots & 4 & 2 & n+1 & 2n & n+3 & \cdots & 2n-2\\
2n-1 & \cdots & n+4 & n+2 & 1 & 3 & 5 & \cdots & n
\end{array}\right].\] Then \[\sum_{t=1}^{mn-1} d( \sigma^{-1}(t),\sigma^{-1}(t+1))=(n+1)^2-4.\]
\end{lemma}

\begin{lemma}\label{lem;888}
Assume that $m$ and $n$ are both even and set $r=mn/2$. Let $\sigma\colon V\to [mn]$ be the bijection given by
\[\left[
\begin{array}{cccc|cccc}
r-1 & \cdots & \cdots & \cdots & \cdots & \cdots & \cdots & mn-1\\
\vdots &  &  & \vdots & \vdots &  &  & \vdots\\
\cdots & \cdots & 3 & 1 & r+1 & r+3 & \cdots & \cdots\\\hline
\cdots & \cdots & \cdots & mn & r & \cdots & \cdots & \cdots\\
\vdots &  &  & \vdots & \vdots &  &  & \vdots\\
r+2 & r+4 & \cdots &\cdots & \cdots & \cdots & 4 & 2
\end{array}\right].\] Then
\[
\sum_{t=1}^{mn-1} d( \sigma^{-1}(t),\sigma^{-1}(t+1))=\frac{mn(m+n)}{2} - 3.\]
\end{lemma}

\begin{lemma}\label{lem;889}
Assume that $m,n\geq 3$ are both odd and set $r=(m+1)(n-1)/2$. Let $\sigma\colon V\to [mn]$ be the bijection given by
$$\left[
\begin{array}{cccc|c|cccc}
r-1&\cdots&\cdots&\cdots&r+1&mn-2&\cdots&\cdots&\cdots\\
\vdots&&&\vdots &\vdots&\vdots&&&\vdots\\
\vdots&&&\vdots&r+m-2&\cdots&\cdots&r+m+2&r+m\\
\cline{5-9}
\cdots&\cdots&3&1&mn&r&\cdots&\cdots&\cdots\\\cline{1-5}
mn-1&\cdots&\cdots&\cdots&r+2&\vdots&&&\vdots\\
\vdots&&&\vdots&\vdots&\vdots&&&\vdots\\
\cdots&\cdots&r+m+3&r+m+1&r+m-1&\cdots&\cdots&4&2
\end{array}\right].$$ Then 
\[
\sum_{t=1}^{mn-1} d( \sigma^{-1}(t),\sigma^{-1}(t+1))=\frac{mn(m+n)}{2}-\frac{m+n}{2}-1.\]
\end{lemma}

\begin{lemma}\label{lem;8899}
Assume that $m\geq 3$ is odd, $n\geq 3$ is even, and set $r=(m+1)n/2-3$. Let $\sigma\colon V\to [mn]$ be the bijection given by
\[\left[
\begin{array}{cccc|cccc}
r &\cdots  &  &\cdots  &  \cdots&  \cdots& r+5 &r+3\\
\vdots &  & & \vdots & \vdots &  &  & \vdots\\
\vdots &  &  & \vdots&  mn-2&\cdots  &\cdots & \cdots \\\cline{6-8}\cline{0-0}
\multicolumn{1}{c|}{mn-1} & \cdots & 3 & 1 & \multicolumn{1}{c|}{mn} & r+1 &\cdots  & \cdots\\\cline{2-5}
mn-3 & \cdots&\cdots & \cdots & \vdots &  &  &\vdots \\
\vdots &  &  & \vdots & \vdots &  & & \vdots\\
\cdots &\cdots  &r+4  &r+2 & \cdots & \cdots & 4 &2
\end{array}\right].\] Then 
\[
\sum_{t=1}^{mn-1} d( \sigma^{-1}(t),\sigma^{-1}(t+1))=\frac{mn(m+n)}{2}-\frac{n}{2}-1.\]
\end{lemma}

\end{document}